\documentclass[11 pt]{amsart}

\usepackage{amsmath}
\usepackage{latexsym,amsfonts,amssymb,mathrsfs}
\usepackage{mathdots}
\usepackage{color}
\usepackage [all,2cell,color]{xy}
\usepackage{enumitem}
\usepackage{tikz-cd}
\UseAllTwocells
\usepackage[
textwidth=3cm,
textsize=small,
colorinlistoftodos]{todonotes}
\usepackage{url}
\usepackage[numbers]{natbib}
\usepackage{hyperref}

\theoremstyle{plain}  
\newtheorem{thm}{Theorem}[section] 
\makeatletter\let\c@thm\c@thm\makeatother
\newtheorem{cor}{Corollary}[section]
\makeatletter\let\c@cor\c@thm\makeatother
\newtheorem{lem}{Lemma}[section]
\makeatletter\let\c@lem\c@thm\makeatother
\newtheorem{prop}{Proposition}[section]
\makeatletter\let\c@prop\c@thm\makeatother

\makeatletter\let\c@claim\c@thm\makeatother

\newtheorem*{unnumberedtheorem}{Theorem}  

\theoremstyle{definition}

\makeatletter\let\c@defn\c@thm\makeatother
\newtheorem{const}{Construction}[section]
\makeatletter\let\c@const\c@thm\makeatother
\newtheorem{notn}{Notation}[section]
\makeatletter\let\c@notn\c@thm\makeatother

\theoremstyle{remark}

\newtheorem{rmk}{Remark}[section]
\makeatletter\let\c@rmk\c@thm\makeatother
\newtheorem{ex}{Example}[section]
\makeatletter\let\c@ex\c@thm\makeatother

\makeatletter\let\c@observationn\c@thm\makeatother

\makeatletter
\let\c@equation\c@thm
\numberwithin{equation}{section}
\makeatother

\usepackage[capitalise]{cleveref}

\newcommand{\newrefformat}[2]{}

\newcommand{\psh}[1]{\set^{#1^{\op}}}

\crefname{lem}{Lemma}{Lemmas}
\crefname{thm}{Theorem}{Theorems}
\crefname{defn}{Definition}{Definitions}
\crefname{notn}{Notation}{Notations}
\crefname{const}{Construction}{Constructions}
\crefname{prop}{Proposition}{Propositions}
\crefname{rmk}{Remark}{Remarks}
\crefname{cor}{Corollary}{Corollaries}
\crefname{equation}{Display}{Displays}
\crefname{ex}{Example}{Examples}

\newcommand{\cC}{\mathcal{C}}

\newcommand{\cS}{\mathcal{S}}

\newcommand{\set}{\cS\!\mathit{et}}

\newcommand{\jbeta}{j_{\beta}}

\newcommand{\Wbar}{T}

\newcommand{\Func}[2]{#2^{#1}}

\DeclareMathOperator{\colim}{colim}
\DeclareMathOperator{\id}{id}
 
\DeclareMathOperator{\Map}{Map}

\DeclareMathOperator{\op}{op}
\DeclareMathOperator{\Dec}{Dec}

\begin{document}
\title{The unit of the total d\'{e}calage adjunction}

\author{Viktoriya Ozornova}
\address{Fakult\"at f\"ur Mathematik, Ruhr-Universit\"at Bochum, Bochum, Germany}
\email{viktoriya.ozornova@rub.de}

\author{Martina Rovelli}
\address{Department of Mathematics, Johns Hopkins University, Baltimore (MD), United States}
\email{mrovelli@math.jhu.edu}

\thanks{The second-named author was partially funded by the Swiss National Science Foundation, grant P2ELP2\textunderscore172086.}

\date{\today}



\thanks{}

\begin{abstract}
We consider the d\'ecalage construction $\Dec$ and its right adjoint $\Wbar$. These functors are induced on the category of simplicial objects valued in any bicomplete category $\cC$ by the ordinal sum. 
We identify $\Wbar\Dec X$ with the path object $X^{\Delta[1]}$ for any simplicial object $X$. We then use this formula to 
produce
an explicit retracting homotopy for the unit $X\to \Wbar\Dec X$ of the adjunction $(\Dec,\Wbar)$. When $\cC$ is a category of objects of an algebraic nature, we then show that the unit is a weak equivalence of simplicial objects in $\cC$.
\end{abstract}
\maketitle

\section{Introduction}
Let $\sigma\colon \Delta\times \Delta \to \Delta$ denote the ordinal sum functor on the simplex category $\Delta$, described on objects via $\sigma([k],[l])=[k+1+l]$.
The induced functor $\sigma^*\colon \psh{\Delta} \to \psh{(\Delta\times\Delta)}$ is sometimes called \emph{total d\'{e}calage} 
and denoted $\Dec$, going back to Illusie \cite{Illusie}. Informally speaking, it spreads out a simplicial set $X$ ``anti-diagonally''  into a bisimplicial set $\Dec X$, which is levelwise described on objects by $(\Dec X)_{k,l}=X_{k+1+l}$. The functor $\sigma^*$ has a right adjoint $\sigma_*$.

This right adjoint often appears
in the literature; it is known as $\overline{W}$, e.g.\ in \cite{CegarraRemediosBar, CegarraRemediosBehaviour}, or as the \emph{total simplicial set} functor $T$,  e.g.\ in \cite{ArtinMazur}, or as the \emph{Artin--Mazur codiagonal}, e.g.\ in \cite{StevensonDecalage}. 
When composed with the levelwise nerve functor, the functor $\sigma_*$ yields a model for the classifying space of simplicial groups, which is also referred to as the $\overline{W}$-construction; see \cite[\textsection5]{StevensonDecalage} for further discussion.

The adjunction
\[
\Dec=\sigma^*\colon \Func{\Delta^{\op}}{\set} \rightleftarrows \Func{(\Delta\times\Delta)^{\op}}{\set}:\sigma_*=T,
\]
plays a crucial role in work by Cegarra, Heredia, Remedios \cite{CegarraRemediosBar,CegarraRemediosBehaviour,CegarraHerediaRemediosDouble} and Stevenson \cite{StevensonDecalage}, in particular due to its relation with Thomason's homotopy colimit formula and Kan's simplicial loop group functor, respectively.
In \cite[Prop.~7.1]{CegarraRemediosBehaviour} and \cite[Lemma 20]{StevensonDecalage}
the two groups of authors prove that the unit $X\to \Wbar\Dec X$ is a weak equivalence for any simplicial set $X$. Both of their proofs rely on the homotopy-theoretical
fact that the Artin--Mazur codiagonal of a bisimplicial set is weakly equivalent to its diagonal.

The aim of this article is to give a full and explicit description of the simplicial object $\Wbar\Dec X$ and of the unit $X\to\Wbar\Dec X$. This description is then used to construct an explicit retracting homotopy for the unit $X\to\Wbar\Dec X$, which is therefore a strong deformation retract.

We work in greater generality than in the above cited articles, and consider the adjunction
\[
\Dec=\sigma^*\colon \Func{\Delta^{\op}}{\cC} \rightleftarrows \Func{(\Delta\times\Delta)^{\op}}{\cC}:\sigma_*=T,
\]
where $\cC$ is any bicomplete category, rather than just the category $\set$ of sets. Examples of interest include the category $\psh{\Delta}$ of simplicial sets (more generally any category of (pre)sheaves valued in a bicomplete category), the category of small categories, and the category of (abelian) groups.
In this more general framework, the d\'ecalage construction $\Dec$ is the underlying bisimplicial space of the \emph{path construction} from \cite{BOORS2}, and the adjoint $T$ is closely related to the \emph{generalized $S_{\bullet}$-construction}.

The category $\Func{\Delta^{\op}}{\cC}$ is always cotensored over $\psh{\Delta}$, and in this paper we prove the following description of the unit, which will appear as \cref{DecUnit}.
\begin{unnumberedtheorem}
For any simplicial object $X$ in $\cC$, there is a natural isomorphism
$$\Wbar\Dec X\cong X^{\Delta[1]}$$
and, under this isomorphism, the unit $X\to \Wbar\Dec X$ of the adjunction $(\Dec,\Wbar)$ is identified with the map $X^{\Delta[0]}\to X^{\Delta[1]}$ induced by $\Delta[1]\to\Delta[0]$.
\end{unnumberedtheorem}

The main ingredient, which appears as \cref{DecCounit}, is a careful analysis of the counit of the adjunction of $\sigma^*$ and its left adjoint $\sigma_!$.  An equivalent description of such counit in the case $\cC=\set$ was already mentioned by Cordier--Porter in \cite{CordierPorterCoherent}.
Similar combinatorics were also considered in \cite[Remark 0.16]{DuskinTriple}, \cite[\textsection 2.9]{EhlersThesis}, \cite[\textsection 4]{EhlersPorterOrdinal}, \cite[\textsection 9.3]{JardineLocal}, \cite[\textsection 3]{AKock},  \cite[\textsection 11.4]{PaoliBook2017}, \cite[\textsection II.5]{StephanThesis} and \cite[\textsection 5.1]{VerityComplicialAMS}.

Moreover, the category $\Func{\Delta^{\op}}{\cC}$ is always enriched over $\psh{\Delta}$ (see e.g.\ \cite{GoerssJardine, portermenagerie,QuillenHA})
and there is therefore a canonical notion of homotopy
between maps in $\Func{\Delta^{\op}}{\cC}$ (cf.~\cite[\textsection II.1.6]{QuillenHA} or \cite[\textsection 3.8]{RiehlCHT}). 
We then produce in \cref{corollarydeformationretract} an explicit retracting homotopy for the unit with respect to this simplicial enrichment, obtaining the following.

\begin{unnumberedtheorem}
For any simplicial object $X$ in $\cC$, the unit $X\to \Wbar\Dec X$ of the adjunction $(\Dec ,\Wbar)$ is strong deformation retract. 
\end{unnumberedtheorem}

To give further homotopical meaning to this strong deformation retract, it is convenient to have a model structure on $\cC^{\Delta^{\op}}$ that is compatible with the simplicial enrichment. For the categories of sets and of many objects of algebraic nature that have a well-behaved forgetful functor to sets (such as groups, modules, or rings), such a model structure was introduced by Quillen (see \cite[\textsection II.4]{QuillenHA} or \cite[\textsection II.5]{GoerssJardine}). 
With respect to this model structure, we prove the following corollary, that will appear as \cref{corollarydeformationretract}.

\begin{unnumberedtheorem}
Let $\cC$ be a category with a suitable forgetful functor to $\set$. For any simplicial object $X$ in $\cC$, the unit $X\to \Wbar\Dec X$ of the adjunction $(\Dec ,\Wbar)$ is a weak equivalence.
\end{unnumberedtheorem}

When specializing to $\cC=\set$, the theorem strengthens and provides a more transparent proof of the known fact that in the classical context the unit of the adjunction $(\Dec ,\Wbar)$ is a weak equivalence (see e.g.~\cite{CegarraRemediosBar}).

\addtocontents{toc}{\protect\setcounter{tocdepth}{1}}

\subsection*{Acknowledgements}
This note was inspired by a joint work Julie Bergner, Ang\'{e}lica Osorno and Claudia Scheimbauer. Moreover, we would like to thank Lennart Meier and Danny Stevenson for useful conversations.
The exposition benefited greatly from the comments of the referee, who also suggested a more general framework for the results.

\section{The main result and applications}

Let $\cC$ be any bicomplete category. By \cite[\textsection II.1]{QuillenHA}, \cite[Thm\ II.2.5]{GoerssJardine} or \cite[Prop.\ 92]{portermenagerie}, 
the category $\Func{\Delta^{\op}}{\cC}$ of simplicial objects in $\cC$ is simplicially enriched\footnote{We warn the reader that, when $\cC$ is itself a simplicial (model) category, there is an alternative simplicial enrichment that is often considered for the category $\cC^{\Delta^{\op}}$ (e.g.\ in \cite{DK,JT,rezk}) and is different from the one in the context of this paper.}.
The mapping spaces assemble into a bifunctor
\[\operatorname{Map}_{\Func{\Delta^{\op}}{\cC}}(-,-)\colon(\Func{\Delta^{\op}}{\cC})^{\op}\times \Func{\Delta^{\op}}{\cC}\to\psh{\Delta}.\]
With respect to this simplicial enrichment the category $\Func{\Delta^{\op}}{\cC}$ is moreover tensored and cotensored over $\psh{\Delta}$, in that there are functors
$$(-)\boxtimes(-)\colon\Func{\Delta^{\op}}{\cC}\times \psh{\Delta}\to\Func{\Delta^{\op}}{\cC}\text{ and }(-)^{(-)}\colon\Func{\Delta^{\op}}{\cC}\times ({\psh{\Delta}})^{\op}\to\Func{\Delta^{\op}}{\cC},$$
with the property that $X \boxtimes -$ is left adjoint to $\operatorname{Map}_{\Func{\Delta^{\op}}{\cC}}(X,-)$ and $(-) \boxtimes K$ is left adjoint to $(-)^K$.
For future reference, we recall that the tensor $X\boxtimes K$ of a simplicial object $X$ with a simplicial set $K$ is given levelwise by the formula
\[
(X\boxtimes K)_k=\coprod\limits_{K_k} X_k,
\]
and there is a canonical map $X\boxtimes K\to X$.

When $\cC$ is the category of sets, groups, rings, or modules over a fixed ring, the structure described above recovers the familiar simplicial enrichments (together with the corresponding tensors and cotensors) for the categories of simplicial sets, simplicial groups, simplicial rings, simplicial modules and simplicial groupoids, which were considered e.g.~in \cite[Ex.~6.2]{GoerssJardine}.

The restriction along the ordinal sum functor $\sigma$ induces a functor
$$\sigma^*\colon\Func{\Delta^{\op}}{ \cC} \to \Func{(\Delta\times\Delta)^{\op}}{ \cC}.$$
that can be computed componentwise as
\[(\sigma^*X)_{m,n}= X_{m+1+n}.\]
By a standard argument (cf.\ e.g.\ \cite[\textsection X.3]{MacLane}), 
the functor $\sigma^*$
has both a left adjoint $\sigma_!$ and a right adjoint $\sigma_*$, given by left and right Kan extensions of a bisimplicial object along $\sigma$. As usual in this type of constructions, $\sigma_!$ is determined by sending representables to representables according to the formula
\[\sigma_!\Delta[k,l]=\Delta[k+1+l],\]
and the value of $\sigma_*$ on a bisimplicial object $Y$ satisfies
\[(\sigma_*Y)_{n}=\Map_{\Func{(\Delta\times\Delta)^{\op}}{\cC}}(\sigma^*\Delta[n],Y).\]

As mentioned in the introduction, when $\cC=\set$
the functor $\sigma^*=\Dec$ is the classical d\'ecalage construction and $\sigma_*=\Wbar$ is the classical $\Wbar$-construction.

We are interested in describing the functor $\sigma_*\sigma^*$ and the unit of the adjunction
\[
\sigma^*\colon \Func{\Delta^{\op}}{\cC} \rightleftarrows \Func{(\Delta\times\Delta)^{\op}}{\cC}:\sigma_*.
\]
For this, we start 
by describing in general the functor $\sigma_!\sigma^*$ and the counit of the adjunction
\[\sigma_!\colon \Func{(\Delta\times\Delta)^{\op}}{\cC} \rightleftarrows \Func{\Delta^{\op}}{\cC}:\sigma^*.\]

\begin{thm}\label{DecCounit}
Let $\cC$ be a bicomplete category.  
 For any simplicial object $X$ in $\cC$, there is a natural isomorphism
 $$\sigma_!\sigma^*X\cong X\boxtimes \Delta[1]$$
 and, under this isomorphism, the counit $\sigma_!\sigma^*X\to X$ is identified with the canonical map $X\boxtimes \Delta[1]\to X$.
 \end{thm}

Before proving \cref{DecCounit}, we discuss a few direct consequences.
First, the theorem can be used to produce an explicit description for the unit of the adjunction $(\sigma^*,\sigma_*)$.

\begin{cor}\label{DecUnit}
Let $\cC$ be a bicomplete category. 
 For any simplicial object $X$ in $\cC$, there is a natural isomorphism
 $$\sigma_*\sigma^*X\cong X^{\Delta[1]}$$
 and, under this isomorphism, the unit $X\to\sigma_*\sigma^*X$ is identified with the map $X^{\Delta[0]}\to X^{\Delta[1]}$ induced by $\Delta[1]\to\Delta[0]$.
\end{cor}

\begin{proof}
The natural isomorphism follows from the fact $\sigma_*\sigma^*$ and $(-)^{\Delta[1]}$ are the right adjoints of the functors $\sigma_!\sigma^*$ and $(-)\boxtimes\Delta[1]$, which are isomorphic by \cref{DecCounit}, and by the uniqueness of right adjoints (dual to \cite[Corollary IV.1.1]{MacLane}).
\end{proof}

Given that the category $\Func{\Delta^{\op}}{\cC}$ is always enriched, tensored and cotensored over $\psh{\Delta}$, there is a canonical notion of homotopy
between maps (cf.~\cite[\textsection II.1.6]{QuillenHA} or \cite[\textsection3.8]{RiehlCHT}) and therefore of strong deformations retracts. Relying on the explicit description of the unit, we can exploit the homotopy theory of simplicial sets and prove the following.

\begin{cor}
\label{corollarydeformationretract}
Let $X$ be a simplicial object in $\cC$.
The unit $X\to\sigma_*\sigma^*X$ is a strong deformation retract.
\end{cor}

\begin{proof}
Thanks to \cref{DecUnit}, it is enough to prove that the map $X\cong X^{\Delta[0]}\to X^{\Delta[1]}$ is a strong deformation retract.
To see this, consider the maps
$$d^1\colon\Delta[0]\to\Delta[1]\text{ and }s^0\colon\Delta[1]\to\Delta[0]$$
which satisfy $s^0\circ d^1=\id_{\Delta[0]}$ and $d^1\circ s^0\simeq_l\id_{\Delta[1]}$, where the symbol $\simeq_l$ denotes the (non-symmetric) relation of left homotopy of simplicial maps. They induce maps
$$d_1\colon X^{\Delta[1]}\to X^{\Delta[0]}\text{ and }s_0\colon X^{\Delta[0]}\to X^{\Delta[1]}$$
which satisfy the relations $d_1\circ s_0=\id_{X^{\Delta[0]}}$ and $s_0\circ d_1\simeq_l\id_{X^{\Delta[1]}}$. This completes the proof.
\end{proof}

Having in mind categories of an algebraic nature (such as those of groups, rings and modules), Quillen identifies certain conditions on a category $\cC$ so that the category $\cC^{{\Delta}^{\op}}$ supports a model structure that is compatible with the simplicial enrichment. Quillen's original result is \cite[Thm\ II.4.4]{QuillenHA}, but we here recall the formulation from \cite[Thm\ 5.1 \& 5.4]{GoerssJardine}.

\begin{thm}
Let $\cC$ be a bicomplete category, and $U\colon\cC^{{\Delta}^{\op}}\to\set^{{\Delta}^{\op}}$ a functor that admits a left adjoint and that respects filtered colimits.
Then the category $\cC^{{\Delta}^{\op}}$ supports a simplicial model structure in which a morphism $f\colon X\to X'$ is a fibration (resp. weak equivalence) if and only if $Uf\colon UX\to UX'$
is a fibration (resp. weak equivalence) in the Kan--Quillen model structure, provided that every map with the left lifting property with respect to all fibrations is a weak equivalence.
\end{thm}

In particular, this model structure recovers the Kan--Quillen model structure, and the usual model structure for simplicial objects of algebraic nature, as recalled in  \cite[Ex.\ 6.2]{GoerssJardine}. The model structure for simplicial commutative rings and for simplicial commutative algebras over a commutative ring are used e.g.~in \cite{mathew,tv}.

With respect to this model structure, we obtain the following corollary.

 \begin{cor}
\label{corollaryweakequivalence}
 For any simplicial object $X$, the unit $X\to\sigma_*\sigma^*X$ is a weak equivalence in $\cC^{\Delta^{\op}}$.
 \end{cor}

\begin{proof}
The statement follows from the fact that in a simplicial model structure all deformation retracts are weak equivalences by \cite[Prop.\ 9.5.16]{Hirschhorn}.
\end{proof}

When $\cC=\set$ and the category $\set^{\Delta^{\op}}$ is endowed with the Kan--Quillen model structure, the corollary specializes to the well-known \cite[Prop.\ 7.1]{CegarraRemediosBar}, and it seems to be new in its generality.

\section{The proof of the main result}

To prove the theorem, we will use the relation of $\Delta$ with the category $\Delta_a$, which is the category $\Delta$ with an additional initial object $[-1]=\varnothing$.
The ordinal sum $\sigma_a$ makes sense as a functor
$\sigma_a\colon \Delta_a\times \Delta_a\to \Delta_a$
and endows $\Delta_a$ with a monoidal structure, whose unit object is the new object $[-1]$.
 We will denote the inclusion of $\Delta$ into $\Delta_a$ by $\iota\colon \Delta \to \Delta_a$.
 
\begin{rmk}
\label{additionalmaps} 
Given a map $\beta\colon[l]\to[k]$ in $\Delta_a$, either $l\neq-1$ and $\beta$ actually lives in $\Delta$, or $l=-1$ and $\beta$ can be written as a composite of the unique map $[-1]\to[0]$ with a map $[0]\to[k]$ in $\Delta$.

In particular, a presheaf $X\in\Func{\Delta_a^{\op}}{\cC}$ is an \emph{augmented simplicial object} in $\cC$, and to specify the structure of $X$ it is enough to specify its structure as a simplicial object $X\in\Func{\Delta^{\op}}{\cC}$, together with an extra face map $d_0\colon X_0\to X_{-1}$ that coequalizes all the other structure maps, i.e., it satisfies the extra simplicial identity $d_0d_0=d_0d_1\colon X_1\to X_{-1}$.

Similarly, to specify the structure map of a presheaf $Y\in\Func{(\Delta_a\times\Delta_a)^{\op}}{\cC}$ it is enough to specify its structure as a bisimplicial object $Y\in\Func{(\Delta\times\Delta)^{\op}}{\cC}$, together with the additional structure maps for the objects that involve $[-1]$. 
\end{rmk}

Since $\cC$ is cocomplete, the functors
$$(\iota\times \iota)^*\colon \Func{\Delta_a^{\op}\times \Delta_a^{\op}}{\cC}\to \Func{\Delta^{\op}\times \Delta^{\op}}{\cC},$$
$$\text{and }\sigma_a^*\colon \Func{\Delta_a^{\op}}{\cC}\to \Func{\Delta_a^{\op}\times \Delta_a^{\op}}{\cC}$$ both admit left adjoints $(\iota\times\iota)_!$ and $(\sigma_a)_!$, which can be used to describe $\sigma_!$. Indeed, it is pointed out in \cite[\textsection2]{StevensonParametrized} that for every simplicial object $X$ there is a natural isomorphism
$$\sigma_!(X)\cong\iota^*(\sigma_a)_!(\iota\times \iota)_!(X).$$

We give an explicit formula for $(\iota\times\iota)_!$ that makes use of the construction $\pi_0$, a generalization of the set of connected components of a simplicial set.

\begin{notn}
\label{definitionpi0}
 Let $\cC$ be a category with coequalizers. For a simplicial object $X$ in $\cC$, we denote by
 \[
  \pi_0(X):=\colim\left(X_1\rightrightarrows X_0\right)
 \]
the coequalizer of the face maps $d_1,d_0\colon X_1\to X_0$. This defines a functor
\[
\pi_0\colon \Func{\Delta^{\op}}{\cC}\rightarrow \cC.
\]
\end{notn}

We can now give a formula for $(\iota\times\iota)_!$, which was known to experts.
Note that,
for any bisimplicial object $Y$ in $\cC$, the assignment $[k]\mapsto \pi_0(Y_{k,-})$ defines a simplicial object in $\cC$, to which $\pi_0$ can be applied again. The result is denoted by $\pi_0\pi_0(Y)$.  

\begin{prop}\label{iotatimesiota!}
  Let $\cC$ be a category with coequalizers. The left adjoint
  $$(\iota\times\iota)_!\colon \Func{\Delta^{\op}\times \Delta^{\op}}{\cC}\to \Func{\Delta_a^{\op}\times \Delta_a^{\op}}{\cC}$$ of the precomposition functor $(\iota\times \iota)^*$
  is given levelwise on objects by
\[
 ((\iota\times\iota)_!Y)_{i,i'}=\begin{cases}
               \pi_0(\pi_0(Y)), \mbox{ if }i=i'=-1,\\
		\pi_0(Y_{i,-}), \mbox{ if }i>-1, i'=-1,\\
		\pi_0(Y_{-,i'}),\mbox{ if } i=-1, i'>-1,\\
		Y_{i,i'}, \mbox{ else.}
              \end{cases}
\]

The bisimplicial structure of $(\iota\times\iota)_!Y$ is inherited from $Y$, and
the additional structure maps (in the sense of \cref{additionalmaps}) are given by the quotient maps
$$Y_{i,0}\to \pi_0(Y_{i,-})\text{ and }Y_{0,i}\to \pi_0(Y_{-,i})$$
and the maps induced on $\pi_0$. 
\end{prop}

The proof of the proposition needs two preliminary lemmas. We start by recording the following description of $\iota_!$, the left adjoint of the restriction functor $\iota^*\colon \Func{\Delta_a^{\op}}{\cC}\to \Func{\Delta^{\op}}{\cC}$.

\begin{lem}[{\cite[\textsection2.2]{StevensonParametrized}}]
\label{iota!}
 Let $\cC$ be a category with coequalizers. The left adjoint
 $$\iota_!\colon \Func{\Delta^{\op}}{\cC}\to \Func{\Delta_a^{\op}}{\cC}$$ of the restriction functor $\iota^*$ is given levelwise on objects by
\[
 (\iota_!X)_k=\begin{cases}
               \pi_0(X), \mbox{ if }k=-1,\\
		X_k, \mbox{ else.}
              \end{cases}
\]
The simplicial structure of $\iota_!X$ is inherited from $X$, and
the additional structure map (in the sense of \cref{additionalmaps}) of $\iota_!X$ is given by the quotient map $X_0\to \pi_0(X)$.
\end{lem}

The following lemma is straightforward.

\begin{lem}\label{FunctorCategoriesAdjShriek}
 Let $I_1, I_2, J$ be small categories, $f\colon I_1\to I_2$ a functor, and $\cC$ a cocomplete category. Then the functors
\[
\begin{aligned}
 (f\times \id)_! \colon \Func{I_1\times J}{\cC} \to \Func{I_2\times J}{\cC}\ \mbox{ and }\ 
 f_! \colon \Func{I_1}{(\Func{J}{\cC})}\to \Func{I_2}{(\Func{J}{\cC})}
\end{aligned}
\]
exist and can be identified modulo the natural equivalence of categories
$$\Func{I_i\times J}{\cC} \simeq \Func{I_i}{(\Func{J}{\cC})}.$$
\end{lem}

We can now prove the formula for $(\iota\times\iota)_!$.

\begin{proof}[Proof of \cref{iotatimesiota!}]
 Given a bisimplicial object $Y$, the desired formula for $(\iota\times\iota)_!(Y)$ is a straightforward computation, using the identification
 $$(\iota\times\iota)_!(Y)=((\iota\times \id)\circ(\id \times \iota))_!(Y)=(\iota\times \id)_!\circ (\id \times \iota)_!(Y),$$
 together with \cref{FunctorCategoriesAdjShriek} and \cref{iota!}.
\end{proof}

We now proceed to the study of $({\sigma_a})_!$. The following construction (and the observation that it is always possible) will be crucial for the rest of the argument. Note that the object $[-1]$ plays an essential role, and the same construction would not make sense when replacing $\Delta_a$ with $\Delta$.
\begin{const}
\label{betais}
Let $\beta \colon [l] \to [k]$ be a morphism in $\Delta_a$, and let $-1\leq i\leq k$. Then there is a unique number $-1\leq \jbeta(i) \leq l$
and a unique pair of morphisms $ (\beta_{1,i}, \beta_{2,i})$ in $\Delta_a\times \Delta_a$ with
\[
 \beta_{1,i}\colon[\jbeta(i)]\to[i]\quad\text{ and } \quad\beta_{2,i}\colon[l-\jbeta(i)-1]\to[k-i-1]
\]
so that $\sigma_a(\beta_{1,i}, \beta_{2,i})=\beta$.
Moreover, the value of $\jbeta(i)$ is given by
 $$
\jbeta(i)=\left\{\begin{array}{lrc}
 \max \{j \in [l]\,|\,\beta(j)\leq i\}&\text{if }\{j \in [l]\,|\,\beta(j)\leq i\}\neq\varnothing,&\\ 
  -1&\text{if }\{j \in [l]\,|\,\beta(j)\leq i\}=\varnothing.&\\
 \end{array}
 \right.
$$

Indeed, if there exists $(\alpha_1,\alpha_2)$
 in $\Delta_a\times \Delta_a$ with
\[
\alpha_1\colon[l_1]\to[i]\quad\text{ and } \quad\alpha_2\colon[l_2]\to[k-i-1]
\]
 such that $\sigma_a(\alpha_1,\alpha_2)=\beta$, we necessarily have 
$$\beta(j)\leq i\text{ for }j\leq l_1\text{ and }\beta(j)>i\text{ for }j>l_1.$$
This forces the value of $l_1$ to be
\[
 l_1=\left\{\begin{array}{lrc}
 \max \{j \in [l]\,|\,\beta(j)\leq i\}&\text{if }\{j \in [l]\,|\,\beta(j)\leq i\}\neq\varnothing,&\\
  -1&\text{if }\{j \in [l]\,|\,\beta(j)\leq i\}=\varnothing.&\\
 \end{array}
 \right.
\]
If we define the maps
$$\beta_{i,1}\colon[\jbeta(i)]\to[i]\text{ and }\beta_{i,2}\colon[l-\jbeta(i)-1]\to[k-i-1]$$
by means of the assignments
$$r\mapsto\beta(r)\quad\text{ and }\quad r\mapsto\beta(\jbeta(i)+1+r)-i-1,\text{ respectively,}$$ 
then one can check that $\sigma_a(\beta_{1,i}, \beta_{2,i})=\beta$ and the pair $(\beta_{i,1},\beta_{i,2})$ is unique with this property.

\end{const}

\begin{ex}
The coface map $d^k \colon[n]\to[n+1]$ can be written as\linebreak$\sigma_a(\id_{[i]},d^{k-i-1})$ for $i<k$ and $\sigma_a(d^k, \id_{[n-i-1]})$ for $i\geq k$.
\end{ex}
We can now give the formula for $(\sigma_a)_!$, cf.~\cite[Chapter 3]{Joyal} and \cite[Lemma 5.1]{StephanThesis}. 

\begin{prop}\label{sigmaa!}
 Let $\cC$ be a cocomplete 
 category. The left adjoint
 $$(\sigma_a)_! \colon \Func{\Delta_a^{\op}\times \Delta_a^{\op}}{\cC}\to\Func{\Delta_a^{\op}}{\cC}$$
 of $(\sigma_a)^*$ is given levelwise on objects by
\[
 ((\sigma_a)_!A)_k=\coprod_{i=-1}^{k} A_{i,k-i-1}.
\]
The structure map of the simplicial set $(\sigma_a)_!A$ induced by some $\beta \colon [l] \to [k]$ is given on the $i$-th summand via the structure map 
of $A$ induced by $(\beta_{1,i}, \beta_{2,i})$,
$$A_{i,k-i-1}\to A_{\jbeta(i),l-\jbeta(i)-1}\to\coprod\limits_{j=-1}^{l} A_{j,l-j-1}= ((\sigma_a)_!A)_l,$$
where $\beta_{1,i}\colon[\jbeta(i)]\to[i]$ and $\beta_{2,i}\colon[l-\jbeta(i)-1]\to[k-i-1]$
are the morphisms described in \cref{betais} and uniquely determined by the condition $\sigma_a(\beta_{1,i}, \beta_{2,i})=\beta$.
\end{prop}

The proof will make use of the following description of the slice category $[k]\downarrow\sigma_a$.

\begin{rmk}
\label{descriptioncomma}
Observe that the set of objects of $[k]\downarrow\sigma_a$
can be canonically identified with
$$\{(l_1,l_2, \gamma)\ |\ l_1,l_2\ge-1,\gamma \colon [k]\to\sigma_a([l_1],[l_2])=[l_1+1+l_2]\text{ in }\Delta_a\}.$$
Modulo this identification, a morphism in $[k]\downarrow\sigma_a$ from $(l_1',l_2',\gamma')$ to $(l_1, l_2, \gamma)$ consists of a pair of morphisms $(\alpha_1,\alpha_2)$ in $\Delta_a\times \Delta_a$ 
with
\[
\alpha_1\colon[l_1']\to[l_1]\quad\text{ and } \quad\alpha_2\colon[l_2']\to[l_2]
\]
such that $\sigma_a([\alpha_1],[\alpha_2])\circ \gamma'=\gamma$.
\end{rmk}

For any $k\ge-1$, we regard the set $\{-1,0,1,\dots,k\}$ as a discrete category.
\begin{lem}
\label{finality}
The functor
$$J\colon\{-1,0,1,\dots,k\}\to[k]\downarrow\sigma_a\quad\text{ given by }\quad j\mapsto (j, k-j-1, \id_{[k]})$$
is initial.
Equivalently, its opposite functor
$$J\colon\{-1,0,1,\dots,k\}\to([k]\downarrow\sigma_a)^{\op}\cong\sigma_a^{\op}\downarrow[k]$$
is final.
\end{lem}

\begin{proof}[Proof of \cref{finality}]

Given an object $(i_1,i_2,\gamma)$ of $[k]\downarrow\sigma_a$, by \cref{betais} there exists a unique 
pair of morphisms $(\alpha_1,\alpha_2)$ in $\Delta_a\times \Delta_a$ 
with
\[
\alpha_1\colon[\jbeta(i_1)]\to[i_1]\quad\text{ and } \quad\alpha_2\colon[k-\jbeta(i_1)-1]\to[i_2]
\]
such that $\sigma_a([\alpha_1],[\alpha_2])=\gamma$. In other words, there exists a unique $j$, 
with $-1\leq j\leq k$, and a unique morphism in $[k]\downarrow\sigma_a$ to $(i_1,i_2,\gamma)$ from an element of the form $J(j)=(j, k-j-1, \id_{[k]})$. 
This proves that, for any object $(i_1,i_2,\gamma)$ of $[k]\downarrow\sigma_a$, the slice category $J\downarrow(i_1,i_2,\gamma)$
has precisely one object.
In particular, the slice category $J\downarrow(i_1,i_2,\gamma)$ is connected and not empty. Thus
$$J\colon\{-1,0,1,\dots,k\}\to[k]\downarrow\sigma_a$$
is initial.
\end{proof}

We can now prove the proposition.

\begin{proof}[Proof of \cref{sigmaa!}]
We deduce the formula for $((\sigma_a)_!A)_k $ by means of the pointwise left Kan extension formula from \cite[Theorem X.3.1]{MacLane},
the cofinality of $J$ from \cref{finality}, and the key property of cofinal functors from \cite[Theorem IX.3.1]{MacLane}: 
$$
\begin{array}{rcl}
 ((\sigma_a)_!A)_k &\cong &\colim \left(\sigma_a^{\op}\downarrow[k]\cong([k]\downarrow\sigma_a)^{\op}\to(\Delta_a\times \Delta_a)^{\op}\xrightarrow{A} \cC\right)\\
&\cong&\colim \left(\{-1,0,\dots,k\}\xrightarrow{J} (\sigma_a\downarrow [k])^{\op} \to(\Delta_a\times \Delta_a)^{\op} \xrightarrow{A} \cC\right)\\
&\cong&\coprod\limits_{j=-1}^{k} A_{j,k-j-1}.
\end{array}
$$
We now describe the structure map 
$$((\sigma_a)_!A)_{k}\to((\sigma_a)_!A)_{l}$$
of $(\sigma_a)_!A$
induced by a map $\beta\colon[l]\to[k]$ in $\Delta_a$ under the chain of isomorphisms above.
By \cite[Theorem X.3.1]{MacLane}, the map induced by $\beta$ 
$$\underset{(i_1,i_2,\gamma)\in\sigma_a^{\op}\downarrow[k]}{\colim}A_{i_1,i_2}\to\underset{(j_1,j_2,\delta)\in\sigma_a^{\op}\downarrow[l]}{\colim}A_{j_1,j_2}$$
identifies the copy of $A_{i_1,i_2}$ corresponding to the component $(i_1,i_2,\gamma)$ of the left-hand side with the copy of $A_{i_1,i_2}$ corresponding to the component $(i_1,i_2,\gamma\circ \beta)$ of the right-hand side. 
By reindexing the colimits according 
to \cite[Theorem IX.3.1]{MacLane} and \cref{finality}, one can check that the map induced by $\beta$ on the sums 
$$\coprod\limits_{i=-1}^{k} A_{i,k-i-1}\cong\underset{i\in\{-1,0,\dots,k\}}{\colim}A_{i,k-i-1}\to\underset{j\in\{-1,0,\dots,l\}}{\colim}A_{j,l-j-1}\cong \coprod\limits_{j=-1}^{l} A_{j,l-j-1}$$
is induced by $(\beta_{i,1},\beta_{i,2})$ on the $i$-th summand $A_{i,k-i-1}$.
\end{proof}

We collect the insights so far to obtain a formula for $\sigma_!$.
\begin{prop}\label{sigma!}
 Let $\cC$ be a cocomplete category.
 The left adjoint
 $$\sigma_! \colon \Func{\Delta^{\op}\times \Delta^{\op}}{\cC}\to\Func{\Delta^{\op}}{\cC}$$
 of $\sigma^*$ is given levelwise on objects by
\[
 (\sigma_!Y)_k\cong \pi_0(Y_{-,k}) \sqcup \coprod_{i=0}^{k-1} Y_{i, k-i-1} \sqcup \pi_0(Y_{k,-}). 
\]

The structure map of the simplicial object
$\sigma_!Y$ induced by some $\beta \colon [l] \to [k]$  in $\Delta$ is given on the first summand of $(\sigma_!Y)_k$ by the maps induced on $\pi_0$, i.e.,
$$\pi_0(Y_{k,-})\to \pi_0(Y_{l,-})\to \pi_0(Y_{-,l}) \sqcup \coprod_{j=0}^{l-1} Y_{j, l-j-1} \sqcup \pi_0(Y_{l,-}),$$
and is given on the last summand of $(\sigma_!Y)_k$ by the dual map induced on $\pi_0$.

To describe the structure map induced by the same $\beta\colon [l]\to [k]$ on the $i$-th summand $Y_{i,k-i-1}$ occurring in $(\sigma_!Y)_k$, decompose $\beta$ as a map of $\Delta_a$ using \cref{betais} as $\beta=\sigma_a(\beta_{i,1},\beta_{i,2})$ with
$$\beta_{1,i}\colon[\jbeta(i)]\to[i]\text{ and }\beta_{2,i}\colon[l-\jbeta(i)-1]\to[k-i-1].$$
If $-1<\jbeta(i)<l$, the map $\beta$ acts on $Y_{i,k-i-1}$ as the structure map of $Y$ corresponding to $(\beta_{1,i},\beta_{2,i})$,
$$Y_{i,k-i-1}\to Y_{\jbeta(i),l-\jbeta(i)-1}\to \pi_0(Y_{-,l}) \sqcup \coprod_{j=0}^{l-1} Y_{j, l-j-1} \sqcup \pi_0(Y_{l,-}).$$
If $\jbeta(i)=-1$, the map $\beta$ acts on $Y_{i,k-i-1}$ as the structure map of $Y$ corresponding to $(\id_{[i]},\beta_{2,i})$ composed with the map onto $\pi_0$,
$$Y_{i,k-i-1}\to Y_{i,l}\to Y_{0,l}\to  \pi_0(Y_{-,l}) \to \pi_0(Y_{-,l}) \sqcup \coprod_{j=0}^{l-1} Y_{j, l-i-1} \sqcup \pi_0(Y_{l,-}).$$
Dually, if $\jbeta(i)=k$, the map $\beta$ acts on $Y_{i,k-i-1}$ as the structure map of $Y$ corresponding to $(\beta_{i,1},\id_{[k-i-1]})$ composed with the map onto $\pi_0$.

Under the identification above, $\sigma_!$ acts on a bisimplicial map as a sum of the corresponding components together with the induced maps on $\pi_0$.
\end{prop}

\begin{proof}
Given a bisimplicial object $Y$, we use the identification
 $$\sigma_!(Y)\cong\iota^*(\sigma_a)_!(\iota\times \iota)_!(Y)$$
from \cite[\textsection2]{StevensonParametrized}, together with \cref{sigmaa!} and \cref{iotatimesiota!},
to obtain the isomorphisms for any $k\ge0$
$$
\sigma_!(Y)_k\cong\iota^*(\sigma_a)_!(\iota\times \iota)_!(Y)_k\cong(\sigma_a)_!(\iota\times \iota)_!(Y)_k\cong\coprod_{i=-1}^{k} (\iota\times \iota)_!(Y)_{i,k-i-1},
$$
where
\[
 (\iota\times \iota)_!(Y)_{i,k-i-1}=\left\{\begin{array}{ll}
		\pi_0(Y_{i,-})=\pi_0(Y_{k,-}), &\mbox{ if }i=k,\\
		\pi_0(Y_{-,k-i-1})=\pi_0(Y_{-,k}),&\mbox{ if } i=-1,\\
		Y_{i,k-i-1} &\mbox{ if }-1<i<k.
              \end{array}\right.
\]

We now describe the structure map
$$\sigma_!(Y)_k\to\sigma_!(Y)_l$$
of $\sigma_!(Y)$
induced by a map $\beta\colon[l]\to[k]$ in $\Delta$ under the chain of isomorphisms above, by saying how it acts on every summand of $\sigma_!(Y)_k$.
By \cref{sigma!} and \cref{identificationpi0}, the map $\beta^*$ acts on the $i$-th summand of $\sigma_!(Y)_k$ as the structure map of $(\iota\times \iota)_!(Y)$ induced by $(\beta_{1,i},\beta_{2,i})$,
\[
 (\beta_{1,i},\beta_{2,i})^* \colon  \left((\iota\times\iota)_!Y\right)_{i, k-i-1}\to \left((\iota\times\iota)_!Y\right)_{\jbeta(i), l-\jbeta(i)-1},
\]
where
$ \beta_{1,i}\colon[\jbeta(i)]\to[i]$ and $\beta_{2,i}\colon[l-\jbeta(i)-1]\to[k-i-1]$
are so that $\sigma_a(\beta_{1,i},\beta_{2,i})=\beta$, as in \cref{betais}. To further rewrite these, we need to distinguish several cases. Factoring $(\beta_{1,i},\beta_{2,i})=(\beta_{1,i},\id)\circ (\id,\beta_{2,i})$, we may assume that $k-i=l-\jbeta(i)$ and $\beta_{2,i}=\id_{[k-i-1]}$, since the two parts can be treated with similar arguments. By means of \cref{iotatimesiota!}, we describe the map
\[
 (\beta_{1,i},\id)^* \colon  \left((\iota\times\iota)_!Y\right)_{i, k-i-1}\to \left((\iota\times\iota)_!Y\right)_{\jbeta(i), k-i-1}.
\]
by distinguishing several cases.
\begin{itemize}
	\item If $-1<i<k$ and  $-1<\jbeta(i)<l$, we obtain
precisely the structure map of $Y$ induced by $(\beta_{1,i},\id)$,
\[
 (\beta_{1,i},\id)^* \colon  Y_{i, k-i-1}\to Y_{\jbeta(i), k-i-1}.
\]
	\item If $i=-1$, then $\jbeta(i)=-1$ and $\beta=\id$. We therefore obtain the identity map,
	\[
\id\colon \pi_0(Y_{-,k})\to \pi_0(Y_{-,k}).
\] 
 	\item If $i=k$, we obtain the map induced on $\pi_0$,
\[
\beta^*\colon \pi_0(Y_{k,-})\to \pi_0(Y_{l,-}).
\] 

	\item If $-1<i<k$ and  $\jbeta(i)=-1$, we obtain the quotient map,
	\[
Y_{i,k}\to\pi_0(Y_{-,k}).
\] 
	\item If $-1<i<k$ and  $\jbeta(i)=l$, we obtain the quotient map composed with $\beta^*$,
		\[
Y_{i,k}\to\pi_0(Y_{l,-}).\qedhere
\]
\end{itemize}
\end{proof}

We will need the following property of simplicial objects, which is a variant of \cite[Lemma 89]{VerityComplicialAMS}.

\begin{lem}\label{CoeqSimplicial}
 Let $\cC$ be any category. For any a simplicial object $X$ in $\cC$ the diagram
\[
 \begin{tikzcd}
   X_k \arrow[r, shift left, "d_i"] \arrow[r, shift right, "d_{i+1}" swap]  & X_{k-1} \arrow[r, "d_i"]  & X_{k-2}
 \end{tikzcd}
\]
is a coequalizer diagram in $\cC$ for all $k\geq 2$ and all $0\leq i\leq k-1$.
\end{lem}

\begin{proof}
As a consequence of the simplicial identities for $X$, the diagram \[
 \begin{tikzcd}
   X_k \arrow[r, shift left, "d_i"] \arrow[r, shift right, "d_{i+1}" swap]  & X_{k-1} \arrow[r, "d_i"]  & X_{k-2}
 \end{tikzcd}
\]
becomes a \emph{split fork} (in the sense of 
\cite[\textsection VI.6]{MacLane})
when considered together with the maps
\[
 X_k\xleftarrow{s_{i+1}} X_{k-1} \xleftarrow{s_i} X_{k-2}
\]
when $i<k-1$ and together with the maps
\[
 X_k\xleftarrow{s_{k-2}} X_{k-1} \xleftarrow{s_{k-2}} X_{k-2}
\]
when $i=k-1$. 
By \cite[Lemma VI.6]{MacLane}, the original diagram is therefore a split coequalizer, and in particular a coequalizer.
\end{proof}

We now use this lemma to identify $\pi_0((\sigma^*X)_{k,-})$ and $\pi_0((\sigma^*X)_{-,k})$ with $X_{k}$.

\begin{prop}
\label{identificationpi0}
Let $\cC$ be a cocomplete category. For any simplicial object $X$ in $\cC$ there are isomorphisms
$$
\pi_0((\sigma^*X)_{-,k})
\cong X_{k}\cong\pi_0((\sigma^*X)_{k,-}),$$
which are natural in $X$ and $k$.
\end{prop}

\begin{proof}
By definition of $\sigma^*$, \cref{CoeqSimplicial} and \cref{definitionpi0}, we have the following isomorphisms
$$X_k\cong\colim\left(\begin{tikzcd}
 X_{k+2} \arrow[r, shift left, "{d_0}"] \arrow[r, shift right, "{d_1}" swap] & X_{k+1}
\end{tikzcd}\right)=\quad\quad\quad\quad\quad\quad\quad\quad$$
$$\quad\quad\quad\quad=\colim\left(\begin{tikzcd}
 (\sigma^*X)_{1,k} \arrow[r, shift left, "{(d^0,\id)^*}"] \arrow[r, shift right, "{(d^1,\id)^*}" swap] & (\sigma^*X)_{0,k}
 \end{tikzcd}\right)\cong\pi_0((\sigma^*X)_{-,k}),$$
 which can be checked to be natural in $X$ and $k$ by direct inspection.

A similar argument applies to the second isomorphism.
\end{proof}

We are now ready to prove \cref{DecCounit}, and describe the counit of the adjunction $(\sigma_!,\sigma^*)$. 

\begin{proof}[Proof of \cref{DecCounit}] 
For any simplicial object  $X$ in the bicomplete category $\cC$, we start by identifying $\sigma_!\sigma^*X$.

By means of \cref{iotatimesiota!}, \cref{sigma!}, \cref{identificationpi0} and \cref{CoeqSimplicial}, we obtain the identification
\begin{equation}
\label{formula}
\sigma_!(\sigma^*X)_k\cong\pi_0(\sigma^*X_{-,k})\amalg\coprod_{i=0}^{k-1} \sigma^*X_{i, k-i-1}\amalg\pi_0(\sigma^*X_{k,-})\cong\coprod_{i=-1}^{k} X_k.
\end{equation}

We now describe the structure map
$$(\sigma_!(\sigma^*X))_{k}\to(\sigma_!(\sigma^*X))_{l}$$
of $\sigma_!(\sigma^*X)$
induced by a map $\beta\colon[l]\to[k]$ in $\Delta$ under the chain of isomorphisms above.
By \cref{sigma!} and \cref{identificationpi0},
the map induced by $\beta$ on
$$\coprod_{i=-1}^{k} X_k\to\coprod_{j=-1}^{l} X_l$$
acts on the $i$-th summand $X_k$ occurring in $(\sigma_!(\sigma^*X))_k$ as the structure map of $X$ corresponding to $\beta$ with values in the $\jbeta(i)$-th copy of $X_l$ occuring in $(\sigma_!(\sigma^*X))_l$.

We now describe the $k$-th component
$$\sigma_!\sigma^*X_k\to X_k$$
of the counit of the adjunction $(\sigma_!,\sigma^*)$ at $X$ under the following chain of isomorphisms
$$
\sigma_!(\sigma^*X)_k\cong(\sigma_a)_!(\iota\times \iota)_!(\sigma^*X)_k\cong\coprod_{i=-1}^{k} (\iota\times \iota)_!(\sigma^*X)_{i,k-i-1}\cong\coprod\limits_{i=-1}^{k} X_k,
$$
which gives an alternative but equivalent description of the identification  (\ref{formula}) of $\sigma_!(\sigma^*X)_k$ with the coproduct $\coprod_{i=-1}^{k} X_k$.
We can expand $(\sigma_a)_!$ further using the pointwise Kan extension formula \cite[Theorem X.3.1]{MacLane}, obtaining
$$
\sigma_!(\sigma^*X)_k\cong \underset{(i_1,i_2,\gamma)\in\sigma_a^{\op}\downarrow[k]}{\colim}X_{i_1+1+i_2}\cong\underset{i\in\{-1,0,\dots,k\}}{\colim}X_{i+1+k-i-1}\cong\coprod\limits_{i=-1}^{k} X_k.$$
By \cite[Theorem X.3.1]{MacLane}, the $k$-th component of the counit as a map
$$\underset{(i_1,i_2,\gamma)\in\sigma^{\op}\downarrow[k]}{\colim}X_{i_1+1+i_2}\to X_k$$
acts on the $(i_1,i_2,\gamma)$-th component as the map $X_{i_1+1+i_2}\to X_k$ induced on $X$ by $\gamma$.
By reindexing the colimits according to \cite[Theorem IX.3.1]{MacLane}, one can check that the $k$-th component of the counit under this identification,
$$\coprod\limits_{i=-1}^{k} X_k\cong\underset{i\in\{-1,0,\dots,k\}}{\colim}X_{i+1+k-i-1}\to X_k,$$
is given by the folding map of $X_k$, as it acts on every copy of $X_k$ as the identity of $X_k$.

We need a further identification, which will allow us to see $\coprod_{i=-1}^{k} X_k$ as the $k$-th object of the tensor $X \boxtimes \Delta[1]$
and discuss the naturality in $k$ of this identification.
Note that there is bijection
$$\{-1,0,\dots,k\}\cong\Delta[1]_k \mbox{ given by } i\mapsto f_i,$$ 
where
\[
f_i\colon[k]\to[1]\text{ for $i=-1,\dots,k$ is given by}\;\; r\mapsto\begin{cases}
          0, \quad r\leq i, \\
	  1, \quad\mbox{ else.}
         \end{cases}
\]
This identification leads to the isomorphism
 \[
 (\sigma_!(\sigma^*X))_k\cong  \coprod\limits_{i=-1}^{k} X_k  \cong \coprod\limits_{\Delta[1]_k} X_k\cong (X\boxtimes \Delta[1])_k,
 \]
 which is seen to be natural in $k$ by direct inspection. 
 Moreover, with respect to this isomorphism,
 the folding map
$$\coprod_{i=-1}^{k} X_k\to X_k$$
 corresponds to the canonical map
 \[
 (X\boxtimes\Delta[1])_k\to X_k\cong (X\boxtimes\Delta[0])_k.
 \]
This proves the desired isomorphism of simplicial objects in $\cC$
$$(\sigma_!(\sigma^*X))\cong X\boxtimes\Delta[1],$$
which also allows us to identify the counit with the canonical map
$$X\boxtimes\Delta[1]\to X\cong X\boxtimes\Delta[0]$$
which is induced by the map $\Delta[1]\to\Delta[0].$
\end{proof}

\bibliographystyle{amsalpha}
\bibliography{refDecalage}

\end{document}